\documentclass[11pt]{article}

\usepackage[T1]{fontenc}
\usepackage{lmodern}
\usepackage{microtype}
\usepackage[a4paper,margin=29mm]{geometry}
\usepackage{amsmath,amssymb,amsthm,mathtools,mathrsfs}
\usepackage{array}
\usepackage{enumitem}
\usepackage{needspace}
\usepackage[hidelinks]{hyperref}
\usepackage[nameinlink,capitalise,noabbrev]{cleveref}
\setlist{itemsep=3pt, topsep=5pt}

\newtheorem{theorem}{Theorem}[section]
\newtheorem{proposition}[theorem]{Proposition}
\newtheorem{lemma}[theorem]{Lemma}

\theoremstyle{definition}
\newtheorem{definition}[theorem]{Definition}

\newcommand{\Cp}{\mathbb C_p}
\newcommand{\Qp}{\mathbb Q_p}
\newcommand{\Fpbar}{\overline{\mathbb F}_p}
\newcommand{\PP}{\mathbf P}

\newcommand{\PGL}{\operatorname{PGL}}
\newcommand{\GL}{\operatorname{GL}}
\newcommand{\red}{\operatorname{red}}
\newcommand{\an}{\mathrm{an}}
\newcommand{\Ber}{\mathrm{Ber}}

\title{A Rigid Analytic Curve Can Have a Dense \(p\)-adic Orbit}
\author{Yu Fu}
\date{}

\begin{document}
\maketitle

\begin{abstract}
Let $\mathcal D$ be a product of Drinfeld upper half spaces and let $J$ be the corresponding product of projective linear groups acting on $\mathcal D$.  We prove that if $Y$ is a reduced, quasi-compact separated strict $\mathbb C_p$-analytic space and $f: Y\to\mathcal D$ is analytic with $\dim Y<\dim\mathcal D$, then \(J f(Y(\mathbb C_p))\) is closed and nowhere dense in $\mathcal D(\mathbb C_p)$.

The quasi-compactness assumption is necessary even when the source is a strict rigid analytic curve, while a quasi-compact perfectoid source can still have an orbit closure with nonempty interior. For every finite extension $F/\mathbb Q_p$ and $d\geq3$, we construct a closed smooth irreducible non-quasi-compact curve in $\Omega_F^{d-1}$ with dense orbit under a fixed diagonal one-parameter subgroup.
We also construct a $\mathbb Q_p$-defined analytic map from the open unit disc to $\mathbf P^2$ with dense $\PGL_3(\mathbb Q_p)$-orbit, and an analytic map from a one-dimensional affinoid perfectoid torus
to $\mathbf P^2$ whose $\PGL_3(\mathbb Q_p)$-orbit closure has nonempty interior.
\end{abstract}



\setcounter{tocdepth}{1}
\tableofcontents

\section{Introduction}
Hecke correspondences on a Shimura variety send algebraic subvarieties to algebraic subvarieties. Extending the usual definition for points \cite[\S2.4]{ChaiHeckeOrbits}, we define the Hecke orbit of an algebraic subvariety to be the union of all its images under Hecke correspondences. A folklore conjecture predicts that the $p$-adic closure of the Hecke orbit of a positive-dimensional proper algebraic subvariety should be nowhere dense in the $p$-adic topology on
$\Cp$-valued points; see \cite[Conjecture~1.3]{FuNowhereDensity}.
In this paper we ask whether an analogous assertion holds for rigid analytic maps and actions of local $p$-adic groups.

Let $G$ be a locally compact $p$-adic group acting on a $\Cp$-analytic space $X$, and let $f:Y\to X$ be an analytic map of $\Cp$-analytic spaces with $\dim Y<\dim X$. The $G$-orbit of the image of $f$ is \[ G f(Y(\Cp))=\bigcup_{g\in G}g f(Y(\Cp)).\]
We ask whether this orbit is nowhere dense in $X(\Cp)$. This is equivalent to whether its classical closure \(\overline{G f(Y(\Cp))}^{\,\mathrm{cl}}\) has empty interior in $X(\Cp)$.

The dimension inequality does not by itself imply nowhere density.
Although each translate is the image of a lower-dimensional analytic space, the union need not be an analytic subspace. Indeed, even the orbit of a projective algebraic curve need not be closed, as the example in \Cref{sec:conventions} shows. The question is therefore whether additional geometric hypotheses can prevent a lower-dimensional analytic family from acquiring interior after saturation by a \(p\)-adic group.

\medskip

Fix an embedding into $\Cp$ for every finite extension of
$\mathbb Q_p$ used below. For such an extension $F/\Qp$ and $d\geq2$, let
\[
 \Omega_F^{d-1}
 =\PP_{\Cp}^{d-1}\setminus
   \bigcup_{H\in(\PP_F^{d-1})^\vee(F)}H_{\Cp}
\]
be Drinfeld's upper half space.

For finite extensions $F_i/\Qp$ and integers $d_i\geq2$, the product \(\mathcal D=\prod_{i=1}^r\Omega_{F_i}^{d_i-1}\) carries an action of \(J=\prod_{i=1}^r\PGL_{d_i}(F_i)\) where each factor of $J$ acts on the corresponding factor of $\mathcal D$.
Our nowhere density result applies to a quasi-compact source which is a rigid space admitting a finite admissible covering by affinoid domains.

\begin{theorem}\label{thm:intro-qc}
Let $Y$ be a reduced quasi-compact separated strict $\Cp$-analytic space, and let $f:Y\to\mathcal D$ be an analytic map. Suppose \(\dim Y<\dim\mathcal D.\) Then $J f(Y(\Cp))$ is closed and nowhere dense in $\mathcal D(\Cp)$.
\end{theorem}

We show that the quasi-compactness assumption cannot be omitted, even for a single Drinfeld upper half space. 

\begin{theorem}\label{thm: Drinfeld example}
Let $F/\Qp$ be a finite extension and let $d\geq3$. Define \(\lambda:F^\times\to\PGL_d(F)\) by \(\lambda(q)=[\operatorname{diag}(q^{d-1},q^{d-2},\ldots,q,1)]\).
There is a closed smooth connected irreducible non-quasi-compact analytic curve $A\subset\Omega_F^{d-1}$ such that
\(\overline{\lambda(F^\times)A(\Cp)}^{\,\mathrm{cl}}
=\Omega_F^{d-1}(\Cp)\).
\end{theorem}

The usage of Drinfeld upper half space is important here.
The \(J\)-action on \(\mathcal D^{\Ber}\) is proper, whereas the action on the full projective space is not.
Properness allows us to restrict the acting elements to a compact subset when both the source image and the target region are compact.
The curve with dense orbit in \Cref{thm: Drinfeld example} lies in the same Drinfeld upper half space as in \Cref{thm:intro-qc}, which shows that quasi-compactness of the source is still necessary even when the group action is proper.

\medskip
We also give two examples whose targets are projective spaces, one arising from an open rigid disc and the other from a quasi-compact affinoid perfectoid torus.  
Let $\mathbb D^-_{\Qp}$ be the rigid open unit disc with coordinate $T$, defined by $|T|<1$.
\begin{theorem}
There are power series $A,B\in\mathbb Q_p[[T]]$ converging for $|T|<1$ such that the analytic morphism
\[
 f:\mathbb D^-_{\mathbb Q_p}\to\mathbf P^2_{\mathbb Q_p},
 \quad f(T)=[A(T):B(T):1],
\]
satisfies
\[
 \overline{\PGL_3(\mathbb Q_p)f(\mathbb D^-_{\mathbb Q_p}(\Cp))}^{\,\mathrm{cl}}
     =\mathbf P^2(\Cp).
\]
Both $A$ and $B$ are unbounded on the open unit disc.
\end{theorem}

The last theorem resume quasi-compactness but its domain is perfectoid rather than a strict rigid space of finite type. We start with the unit torus $|T|=1$ and adjoin compatible $p$-power roots of $T$. Completing the resulting ring of analytic functions gives the perfectoid torus used below.

\begin{theorem}
There exist a connected quasi-compact affinoid perfectoid space $P$, whose underlying adic topological space has dimension one, and an  analytic morphism $g:P\to\mathbf P^2_{\Cp}$ such that \(\overline{\PGL_3(\mathbb Q_p)g(P(\Cp))}^{\,\mathrm{cl}}\) contains a nonempty open subset of $\mathbf P^2(\Cp)$.
\end{theorem}

The construction exploits functions coming from successively deeper levels of the perfectoid toric tower. By evaluating at suitable points and applying diagonal rescalings, we obtain a dense set of limits in the open locus
\[
\{[v:w:1]: |v|=|w|=1\}\subset \mathbf P^2(\mathbb C_p).
\]

We note that the examples in this paper are examples for local analytic group
actions, so they do not disprove the conjecture 
\cite[Conjecture~1.3]{FuNowhereDensity} for the Hecke orbits of algebraic subvarieties of global Shimura varieties.

We also note that Theorem~\ref{thm:intro-qc} is also applied in the companion paper \cite{FuGlobalDrinfeld} to
Varshavsky's globally Drinfeld-uniformized projective unitary Shimura varieties.

\vspace{1em}
\noindent \textbf{Organization of the paper.}  \Cref{sec:conventions} fixes the topology and terminology, gives an example of a nonclosed orbit, and recalls two analytic facts.  \Cref{sec:finite-precision,sec:drinfeld-qc} prove the quasi-compact Drinfeld theorem.  In \Cref{sec:drinfeld-counterexample} we construct the non-quasi-compact Drinfeld curve, and in \Cref{sec:qp-gauss-counterexample} we give the $\Qp$-defined example on the open disc.  \Cref{sec:perfectoid-torus-counterexample} treats the perfectoid example.

\section*{Acknowledgements}
The author is grateful to Ananth Shankar for suggesting that she consider whether there is a rigid analytic analogue of the \(p\)-adic nowhere density conjecture for Hecke orbits many years ago. This question led to the present paper. She also thanks Mingjia Zhang for a conversation about \(p\)-adic orbit closures. The conversation inspired her to investigate orbit closures on flag varieties and $p$-adic period domains more closely.
\section{Classical topology and preliminaries}
\label{sec:conventions}

For a strict rigid analytic space $X$ we write $X^{\Ber}$ for its associated Berkovich space. 
An analytic map $f:Y\to X$ induces a continuous map
$f^{\Ber}:Y^{\Ber}\to X^{\Ber}$; we also write $f$ for this map when its source is clear. Over $\Cp$, the classical $\Cp$-rational points are precisely the type I Berkovich points whose completed residue field is $\Cp$. We denote their set by $X(\Cp)$ and give it the usual nonarchimedean topology induced on affinoid charts by embeddings in polydiscs.
For $E\subset X(\Cp)$, the notation \(\overline E^{\,\mathrm{cl}}\) means closure in this classical topology.

For the perfectoid space in \Cref{sec:perfectoid-torus-counterexample}, we use the adic formalism. In that setting, $X(\Cp)$ denotes the set of morphisms $\operatorname{Spa}(\Cp,\mathcal O_{\Cp})\to X$ over $\Cp$.
The images of these points in a rigid analytic target are classical points of that target, and their closures are taken in the same classical topology just specified.

Let $G$ be a topological group acting continuously on $X$.  For $E\subset X(\Cp)$, define \(G E=\bigcup_{g\in G}gE.\)
We say that the $G$-orbit of $E$ is \emph{nowhere dense} if \(\operatorname{Int}_{X(\Cp)} \overline{G E}^{\,\mathrm{cl}}=\varnothing.\) We say that it is \emph{dense} if the same closure equals $X(\Cp)$.

\medskip
\noindent\textbf{A nonclosed orbit.}
The orbit need not itself be closed, even when the subspace is projective and algebraic.
For example, choose \(\alpha\in\overline{\mathbb Q}_p\setminus\mathbb Q_p\), let \(E=\mathbb Q_p(\alpha)\), and let \(L_\alpha=\{X_0+\alpha X_1=0\}\subset\mathbf P^2_{\Cp}.\)
For \(g_n=\operatorname{diag}(p^n,1,1)\), we have
\(
 g_nL_\alpha=\{X_0+\alpha p^nX_1=0\}.
\)
Choose \(t\in\Cp\setminus E\).  The points
\[
 x_n=[-\alpha p^n:1:t]\in g_nL_\alpha(\Cp)
\]
converge to \(x=[0:1:t]\).  On the other hand, \(x\notin\PGL_3(\mathbb Q_p)L_\alpha(\Cp)\).  Indeed, the coefficients of an equation for any translate of \(L_\alpha\) lie in \(E\) and span a two-dimensional \(\mathbb Q_p\)-vector space.  If such an equation vanished at \(x\), it would have the form
\[
 aX_0+bX_1+cX_2=0,\quad b+ct=0,
\]
with \(a,b,c\in E\).  Since \(t\notin E\), this forces \(b=c=0\), contrary to the dimension of the coefficient span.  Thus \(\PGL_3(\mathbb Q_p)L_\alpha(\Cp)\) is not closed. This is why our question concerns the size of the orbit closure rather than the closedness of the orbit itself.

For the remainder of this section, fix a complete nonarchimedean field $K$ with nontrivial valuation. We consider quasiseparated strict rigid analytic $K$-spaces that admit an admissible affinoid covering of finite type, as in \cite[Theorem~1.6.1]{BerkovichEtale}.  Such a covering is of finite type if each member meets only finitely many other members.  In particular, every quasi-compact separated strict rigid analytic $K$-space satisfies these conditions.  We record the two standard facts needed in the orbit argument.
\begin{proposition}
\label{prop:berkovich-compactness}
Let $Y$ be a separated
strict rigid analytic $K$-space. Then $Y^{\Ber}$ is Hausdorff. If $Y$ is
also quasi-compact, then $Y^{\Ber}$ is compact Hausdorff.
\end{proposition}

\begin{proof}
The Hausdorff assertion is \cite[Proposition~1.4.2]{BerkovichEtale}. If $Y$ is quasi-compact, choose a finite admissible affinoid covering of $Y$. The Berkovich spectrum of each affinoid is compact by \cite[Theorem~1.2.1]{BerkovichSpectral}, so $Y^{\Ber}$ is compact.
\end{proof}

\begin{proposition}
\label{prop:rigid-nullstellensatz}
Assume that \(K\) is algebraically closed. Let $f:Y\to X$ be a morphism of strict rigid analytic $K$-spaces. If $x\in X(K)$ and the Berkovich fiber \((f^{\Ber})^{-1}(x)\) is nonempty, then the analytic fiber \(Y_x=Y\times_X\operatorname{Sp}K\) contains a $K$-rational point. Equivalently, if a $K$-rational point of $X$ lies in the image of $f^{\Ber}$, then it already lies in $f(Y(K))$.
\end{proposition}

\begin{proof}
By the rigid--Berkovich comparison and its compatibility with fiber products, \((f^{\Ber})^{-1}(x)\) is the Berkovich realization of $Y_x$; see \cite[Proposition~1.4.1 and Theorem~1.6.1]{BerkovichEtale}. Hence $Y_x$ is nonempty. Choose a nonempty strict affinoid domain $\operatorname{Sp}A\subset Y_x$. By the rigid Nullstellensatz \cite[Corollary~6.1.2/3]{BGR}, the residue field of every maximal ideal of $A$ is a finite extension of $K$. Since $K$ is algebraically closed, such a residue field is $K$, and hence $Y_x(K)\neq\varnothing$.
\end{proof}

\section{Compact families and reduction}
\label{sec:finite-precision}

Let $X$ be a smooth separated strict $\Cp$-analytic space of pure dimension $n$.
Let $Y$ be reduced quasi-compact separated and strict of dimension $s$. Let $f:Y\to X$ be an analytic morphism with $s<n$.  If $f(Y)$ is a closed analytic subspace of dimension at most $\dim Y$, then $f(Y(\Cp))$ has empty interior in $X(\Cp)$.

A compact subset of a $p$-adic group can have positive dimension as a $p$-adic manifold.  Even when the preceding dimension argument applies to each individual translate of $f(Y)$, it does not apply directly to the union of its translates. The translates are indexed by a compact subset of the acting $p$-adic group, which is not in general a $\Cp$-analytic space, and the union need not be an analytic subspace.

We instead work after reduction. For the projective action considered below, once a target polydisc has been fixed, compactness allows the acting elements to be replaced by finitely many representatives without changing the reductions in that polydisc. We will show that the resulting reductions are contained in finitely many algebraic images of dimension at most $s$. We record this property in the following definition.

Let $B\subset X$ be a strict closed rational coordinate polydisc defined by \( |z_i-b_i|\leq|\lambda_i|,\) where \( 1\leq i\leq n.\) Let \(u_i=\frac{z_i-b_i}{\lambda_i}.\)
Here $b_i\in\Cp$ and $\lambda_i\in\Cp^\times$, and the rescaled coordinates $u_i$ identify $B$ with the closed unit polydisc.
Define the reduction map on classical points by
\[
 \red_B:B(\Cp)\rightarrow\mathbf A^n(\Fpbar),
 \quad
 x\mapsto
 \bigl(\overline{u_1(x)},\ldots,\overline{u_n(x)}\bigr),
\]
where the bar denotes reduction modulo the maximal ideal of $\mathcal O_{\Cp}$.
Let $P$ be a compact metrizable space.  Suppose that for every $a\in P$ we have an analytic map \(f_a:Y\to X.\)

\begin{definition}
\label{def:finite-precision}
The family $(f_a)_{a\in P}$ is \emph{finite on every target polydisc} if for each \(B\) there exist finitely many reduced finite type $\Fpbar$-schemes $\mathcal Y_\nu$ of dimension at most $s$ and morphisms
\(
       \overline f_\nu:\mathcal Y_\nu\to
       \mathbf A^n_{\Fpbar}
\)
such that
\begin{equation*}
 \red_B\left(
   \left(\bigcup_{a\in P}f_a(Y(\Cp))\right)\cap B(\Cp)
 \right)
 \subseteq
 \bigcup_\nu\overline f_\nu(\mathcal Y_\nu(\Fpbar)).
\end{equation*} 
The collection $\mathcal Y_\nu$ may depend on $B$.
\end{definition}

The next lemma combines the preceding finiteness property with joint continuity.  In \cref{prop:compact-projective-family} we will verify both hypotheses for the standard action of projective linear groups on projective spaces considered below, uniformly over each compact set of
group elements.
\begin{lemma}
\label{lem:finite-precision-family}
Assume that the family is finite on every target polydisc in the sense of \Cref{def:finite-precision} and that the map
\begin{align}\label{product map}
\Phi: P\times Y^{\Ber} \to X^{\Ber},
       \quad (a,y) \mapsto f_a(y)
\end{align}
is jointly continuous.

If $s<n$, then \(\bigcup_{a\in P}f_a(Y(\Cp))\) is closed and has empty interior in $X(\Cp)$.
\end{lemma}

\begin{proof}
By \Cref{prop:berkovich-compactness}, the Berkovich space $Y^{\Ber}$ is compact and $X^{\Ber}$ is Hausdorff. Let \(C\) be the image of $\Phi$. Since $P$ is compact and $\Phi$ is jointly continuous, $C$ is compact and hence closed in $X^{\Ber}$.
Every point of \(\bigcup_{a\in P}f_a(Y(\Cp))\) is a classical point in $C$.  Conversely, suppose that $x\in X(\Cp)$ and that \(x=f_a(y)\) for some $a\in P$ and $y\in Y^{\Ber}$.  The Berkovich fiber \(f_a^{-1}(x)\) is nonempty.  By \Cref{prop:rigid-nullstellensatz}, this fiber contains a $\Cp$-point. Hence \(x\in f_a(Y(\Cp)).\)
So \(\bigcup_{a\in P}f_a(Y(\Cp))=C\cap X(\Cp).\) Hence this union is closed in \(X(\Cp)\).

Suppose \(\bigcup_{a\in P}f_a(Y(\Cp))\) has nonempty interior in \(X(\Cp)\). Since \(X\) is smooth, there exists a strict closed rational coordinate polydisc \(B\subset X\) such that \[B(\Cp)\subseteq \operatorname{Int}_{X(\Cp)}(\bigcup_{a\in P}f_a(Y(\Cp))).\]
Since the family \((f_a)_{a\in P}\) is finite on every target polydisc, by assumption there exist reduced finite type \(\Fpbar\)-schemes \(\mathcal Y_\nu\) of dimension at most \(s\) and morphisms \(\overline f_\nu:\mathcal Y_\nu\rightarrow\mathbf A^n_{\Fpbar}\)
such that
\[
\red_B(B(\Cp))=\red_B\bigl(\bigcup_{a\in P}f_a(Y(\Cp))\cap B(\Cp)\bigr) \subseteq \bigcup_\nu \overline f_\nu\bigl(\mathcal Y_\nu(\Fpbar)\bigr).
\]
Since \(\red_B\) is surjective, \(\mathbf A^n(\Fpbar) \subseteq \bigcup_\nu \overline f_\nu\bigl(\mathcal Y_\nu(\Fpbar)\bigr).\)
 For each \(\nu\), the Zariski closure of \(\overline f_\nu\bigl(\mathcal Y_\nu(\Fpbar)\bigr)\) in \(\mathbf A^n_{\Fpbar}\) has dimension at most \(\dim\mathcal Y_\nu\leq s<n\), so it is a proper closed subset of \(\mathbf A^n_{\Fpbar}\).
Since there are only finitely many $\nu$, the union cannot cover $\mathbf A^n(\Fpbar)$, contradiction.
\end{proof}

The next proposition shows that the family of maps \(g\circ f\) with \(g\) ranging over a compact subset of a product of projective linear groups satisfies the two hypotheses of \Cref{lem:finite-precision-family}.

\begin{proposition}
\label{prop:compact-projective-family}
For \(1\leq i\leq r\), let \(F_i/\Qp\) be a finite extension and let \(m_i\geq1\). Let \(X=\prod_{i=1}^r\PP^{m_i}_{\Cp}\) and  \(J=\prod_{i=1}^r\PGL_{m_i+1}(F_i)\),  where \(J\) acts factorwise on \(X\). Let \(P\subset J\) be a compact subset.
Let \(Y\) be a reduced quasi-compact separated strict \(\Cp\)-analytic space, and let \(f:Y\to X\) be an analytic morphism.
For \(g\in P\), let \(f_g=g\circ f\). Then the map
\[
\Psi: P\times Y^{\Ber}\rightarrow X^{\Ber}, \quad (g,y)\mapsto g f(y),
\]
is jointly continuous, and the family \((f_g)_{g\in P}\) is finite on every target polydisc.

Consequently, if \(\dim Y<\dim X\), then \(Pf(Y(\Cp))=\bigcup_{g\in P}g f(Y(\Cp))\) is closed and has empty interior in \(X(\Cp)\).
If \(X'\subset X\) is a \(J\)-stable open analytic subspace and \(f\) factors through \(X'\), then the same conclusions hold in \(X'(\Cp)\).
\end{proposition}

\begin{proof}
The action of \(J\) on \(X\) induces a continuous map \(J\times X^{\Ber}\rightarrow X^{\Ber}.\)
Its composition with \(P\times Y^{\Ber}\rightarrow J\times X^{\Ber}\) that sends \((g,y)\mapsto (g,f(y))\) is the map \((g,y)\mapsto g f(y)=f_g(y)\). This proves joint continuity.

We next prove that the family \((f_g)_{g\in P}\) is finite on every target polydisc. Let \(\operatorname{pr}_i:X\to\PP^{m_i}_{\Cp}\) be the projections. Choose a finite admissible affinoid covering of \(Y\) together with trivializations of each line bundle \((\operatorname{pr}_i\circ f)^*\mathcal O_{\PP^{m_i}_{\Cp}}(1)\) on each member \(U\). Write \((\operatorname{pr}_i\circ f)|_U=[F_{i,0}:\cdots:F_{i,m_i}]\), where \(F_{i,j}\in\mathcal O(U)\) are the homogeneous coordinate functions given by these trivializations.

Cover each \(U\) by the finitely many rational domains defined by
\[
|F_{i,j}|\leq |F_{i,k_i}|, \quad 1\leq i\leq r,\quad 0\leq j\leq m_i,
\]
as \((k_1,\ldots,k_r)\) ranges over the tuples with \(0\leq k_i\leq m_i\).
For each \(i\), the functions \(F_{i,0},\ldots,F_{i,m_i}\) generate the unit ideal, so \(F_{i,k_i}\) is invertible on the corresponding domain.
Dividing each \(F_{i,j}\) by \(F_{i,k_i}\) gives homogeneous coordinate functions whose \(k_i\)-th coordinate is \(1\) and whose other coordinates have absolute value at most \(1\). 

It is enough to work on each of these affinoids separately. Assume that \(Y\) is one of these affinoids.
Fix a strict closed rational coordinate polydisc \(B\subset X\) with rescaled coordinates \(u_1,\ldots,u_n\). Consider the set
\[
I_B=
\left\{
(g,y)\in P\times Y^{\Ber}:g f(y)\in B^{\Ber}
\right\}.
\]
Since \(\Psi\) is continuous and \(B^{\Ber}\) is closed in \(X^{\Ber}\), the set \(I_B=\Psi^{-1}(B^{\Ber})\) is closed in the compact space \(P\times Y^{\Ber}\). Hence \(I_B\) is compact.
For each \((g,y)\in I_B\), choose a rational affinoid domain \(W\) in a product of standard projective affine charts contained in the coordinate chart defining \(B\) such that \(g f(y)\) lies in the interior of \(W^{\Ber}\) and the coordinates \(u_1,\ldots,u_n\) are defined on \(W\). By continuity, there are a relatively open neighborhood \(V\subset P\) of \(g\) and a rational affinoid domain \(U\subset Y\) with \(y\) in the interior of \(U^{\Ber}\) such that \(g'f(U^{\Ber})\subseteq W^{\Ber} \) for every \(g'\in V.\)
Write \(g=(g_1,\ldots,g_r)\). For each \(i\), choose \(M_i\in\GL_{m_i+1}(F_i)\) with \([M_i]=g_i\), and choose indices \(\alpha_i,\beta_i\) such that \((M_i)_{\alpha_i\beta_i}\ne0\). Shrink \(V\) to a compact open neighborhood of \(g\) in \(P\) such that every \(g'=(g_1',\ldots,g_r')\in V\) admits representatives \(M_i(g')\in\GL_{m_i+1}(F_i)\) satisfying \([M_i(g')]=g_i'\) and \((M_i(g'))_{\alpha_i\beta_i}=1.\) Thus we have unique \(M_i(g') \in\GL_{m_i+1}(F_i)\) and the maps \(M_i:V\to\GL_{m_i+1}(F_i)\) are continuous.

Since \(I_B\) is compact, it is covered by finitely many \(V_\nu\times U_\nu^{\Ber}\). Let \(W_\nu\) be the corresponding target domains. Define
\[
\delta_{\nu,j}:V_\nu\times V_\nu\times U_\nu^{\Ber}
\rightarrow\mathbb R_{\geq0},
\quad
\delta_{\nu,j}(g,g',y) =|u_j(gf(y))-u_j(g'f(y))|.
\]
For \(g,g'\in V_\nu\) and \(y\in U_\nu^{\Ber}\), both \(u_j(g f(y))\) and \(u_j(g'f(y))\) are elements of the completed residue field \(\mathcal H(y)\) so \(\delta_{\nu,j}\) is well defined. It is also continuous on \(V_\nu\times V_\nu\times U_\nu^{\Ber}.\)
For a rational function \(R\) in the standard affine coordinates with no poles on \(W_\nu\), the map \(g\mapsto R\circ gf\) is continuous in the uniform norm on \(\mathcal O(U_\nu)\). Indeed, the matrix entries vary continuously, and the pullbacks are obtained from the homogeneous coordinate functions by addition, multiplication, and inversion of units. The functions \(u_j\) can be approximated uniformly by such rational functions on \(W_\nu\).
Since \(gf(U_\nu)\subset W_\nu\) for every \(g\in V_\nu\), the pullbacks converge uniformly on \(U_\nu\) and uniformly in \(g\).
Thus \(g\mapsto u_j\circ gf\) is continuous in the uniform norm.
The continuity of \((h,y)\mapsto |h(y)|\) on \(\mathcal O(U_\nu)\times U_\nu^{\Ber}\) now shows that \(\delta_{\nu,j}\) is continuous.

Since \(\delta_{\nu,j}\) is continuous and \(\delta_{\nu,j}(g,g,y)=0\), compactness of \(V_\nu\) and \(U_\nu^{\Ber}\) gives a compact open subgroup \(H\subset J\) such that \(\delta_{\nu,j}(g,g',y)<1\) whenever \(g,g'\in V_\nu\) lie in the same right coset of \(H\) and \(y\in U_\nu^{\Ber}\). Since there are only finitely many \(\nu\) and \(j\), we can choose the same \(H\) for all of them. Since each \(V_\nu\) meets only finitely many right cosets of \(H\), we can replace it by its nonempty intersections with these cosets and reindex the resulting finite covering.

Choose a representative \(g_\nu\in V_\nu\) for each \(\nu\). Let \(g\in V_\nu\) and \(y\in U_\nu(\Cp)\), and suppose that \(g f(y)\in B(\Cp)\). Then
\(|u_j(g f(y))|\leq1\) for \(1\leq j\leq n\).
Since
\[
\delta_{\nu,j}(g,g_\nu,y)=|u_j(g f(y))-u_j(g_\nu f(y))|<1,
\]
the ultrametric inequality gives \(|u_j(g_\nu f(y))|\leq1.\)
Thus \(g_\nu f(y)\in B(\Cp)\) and reduction modulo
\(\mathfrak m_{\Cp}\) gives
\begin{equation}
\red_B(g f(y))
=
\red_B(g_\nu f(y)).
\label{eq:representative-reduction}
\end{equation}

Let \(D_\nu =U_\nu\cap(g_\nu f)^{-1}(B)\) be the rational affinoid domain of \(U_\nu\) defined by \(|u_j(g_\nu f)|\leq1\) for \(1\leq j\leq n\).
Let \(D_\nu=\operatorname{Sp}A_\nu\) for some \(A_\nu\) and let \(\mathcal Y_\nu=\left( \operatorname{Spec}(A_\nu^\circ/A_\nu^{\circ\circ}) \right)_{\mathrm{red}}.\)
The scheme \(\mathcal Y_\nu\) is of finite type over \(\Fpbar\) and we have \(\dim\mathcal Y_\nu\leq\dim Y \) by \cite[\S6.3.4]{BGR}. The functions \(u_j\circ g_\nu f\) are power-bounded on \(D_\nu\), so their reductions define a morphism \(\overline f_\nu: \mathcal Y_\nu\rightarrow\mathbf A^n_{\Fpbar}.\)

Now let \(g\in P\) and \(y\in Y(\Cp)\) satisfy \(g f(y)\in B(\Cp)\). Choose \(\nu\) such that \((g,y)\in V_\nu\times U_\nu^{\Ber}\). Then \(y\in D_\nu(\Cp)\).
Let \(\widetilde y\in\mathcal Y_\nu(\Fpbar)\) denote the reduction of \(y\). Then \eqref{eq:representative-reduction} gives
\[
\red_B(g f(y))=\red_B(g_\nu f(y))=\overline f_\nu(\widetilde y).
\]
Therefore
\[
\red_B\left(
Pf(Y(\Cp))\cap B(\Cp)
\right)
\subseteq
\bigcup_\nu
\overline f_\nu\bigl(\mathcal Y_\nu(\Fpbar)\bigr).
\]
Since \(B\) is arbitrary, the family \((f_g)_{g\in P}\) is finite on every target polydisc.

The group \(J\) is metrizable, so its compact subset \(P\) is a compact metrizable space. The conclusion then follows from \Cref{lem:finite-precision-family}. The same argument also applies to a \(J\)-stable open analytic subspace \(X'\subset X\).
\end{proof}

\section{The quasi-compact theorem}
\label{sec:drinfeld-qc}

Let \(X_0,\ldots,X_{d-1}\) be the homogeneous coordinates on \(\PP^{d-1}_{\Cp}\).
Since \(\Omega_F^{d-1}\) does not meet the hyperplane \(X_{d-1}=0\), the quotient \(\ell/X_{d-1}\) is analytic on \(\Omega_F^{d-1}\) for every \(F\)-linear form \(\ell=\sum_i a_iX_i\).
For \(x\in(\Omega_F^{d-1})^{\Ber}\), define the norm
\[
N_x(\ell)=\left|\left(\frac{\ell}{X_{d-1}}\right)(x)\right|,
\quad
\ell\in FX_0\oplus\cdots\oplus FX_{d-1},
\]
where the absolute value is taken in \(\mathcal H(x)\). Replacing the denominator \(X_{d-1}\) by another nonzero \(F\)-linear form multiplies \(N_x\) by a positive constant.
The homothety class of \(N_x\) defines the standard continuous \(\PGL_d(F)\)-equivariant retraction to the Bruhat--Tits building
\[
\tau: (\Omega_F^{d-1})^{\Ber} \rightarrow \mathcal B(\PGL_d,F), \quad x\mapsto[N_x];
\]
see \cite[proof of Theorem~1]{BerkovichAutomorphism}.

\begin{lemma}
\label{lem:compact-transporter}
Let \(\mathcal D=\prod_{i=1}^r\Omega_{F_i}^{d_i-1}\) and \(J=\prod_{i=1}^r\PGL_{d_i}(F_i)\).
For compact subsets \(C_0,U\subset\mathcal D^{\Ber}\), the set
\[
S_{C_0,U}=\{g\in J:gC_0\cap U\ne\varnothing\}
\]
is compact.
\end{lemma}

\begin{proof}
We first consider a single factor \(\mathcal D=\Omega_F^{d-1}\) and \(J=\PGL_d(F)\).
The building is locally finite, and the setwise stabilizer of each simplex is compact open.
The compact sets \(\tau(C_0)\) and \(\tau(U)\) therefore meet only finitely many simplices.

For \(g\in S_{C_0,U}\), choose \(c\in C_0\) with \(gc\in U\).
Let \(\sigma\) and \(\eta\) be the unique simplices whose relative interiors contain \(\tau(c)\) and \(\tau(gc)\), respectively.
Equivariance gives \(g\tau(c)=\tau(gc)\), so \(g\sigma=\eta\).
There are only finitely many possible pairs \((\sigma,\eta)\).
For each pair for which an element \(g_{\sigma,\eta}\) satisfies \(g_{\sigma,\eta}\sigma=\eta\), we have
\[
\{g\in J:g\sigma=\eta\} =g_{\sigma,\eta}\operatorname{Stab}(\sigma).
\]
Thus \(S_{C_0,U}\) is contained in a finite union of compact cosets.

For the general product, let \(C_{0,i}\) and \(U_i\) be the images of \(C_0\) and \(U\) under the projection to the \(i\)-th factor.
These images are compact, and
\[
S_{C_0,U} \subseteq \prod_{i=1}^r \{g_i\in\PGL_{d_i}(F_i): g_iC_{0,i}\cap U_i\ne\varnothing\}.
\]
By the single factor argument, the right hand side is contained in a compact subset of \(J\).
Hence \(S_{C_0,U}\) is contained in a compact subset of \(J\).

It remains to prove closedness.
Since \(U\) is closed and the action is continuous, the set
\[
E=\{(g,c)\in J\times C_0:gc\in U\}
\]
is closed.
Since \(C_0\) is compact, the projection \(p:J\times C_0\to J\) is a closed map.
Its image \(p(E)=S_{C_0,U}\) is therefore closed.
Thus \(S_{C_0,U}\) is compact.
\end{proof}

\begin{theorem}
\label{thm:drinfeld-qc}
Let $Y$ be a reduced quasi-compact separated strict $\Cp$-analytic space, and let $f:Y\to\mathcal D$ be analytic. Let \(n=\dim\mathcal D=\sum_i(d_i-1).\) If $\dim Y<n$, then \(J f(Y(\Cp))\) is closed and nowhere dense in $\mathcal D(\Cp)$.
\end{theorem}

\begin{proof}
Let $C_0=f(Y^{\Ber})$. This set is compact by
\Cref{prop:berkovich-compactness} and continuity of $f^{\Ber}$.
We first prove closedness in $\mathcal D(\Cp)$.
Let $x$ be a classical limit of $J f(Y(\Cp))$.
Since the classical topology on $\mathcal D(\Cp)$ is metrizable, choose a sequence \(g_mf(y_m)\to x\), with \(g_m\in J\) and \(y_m\in Y(\Cp)\), and a compact neighborhood $U_x$ of $x$ in $\mathcal D^{\Ber}$.
For all sufficiently large \(m\), we have \(g_mf(y_m)\in U_x\), so \(g_m\) belongs to
\[
P_x=\{g\in J:gC_0\cap U_x\ne\varnothing\},
\]
which is compact by \Cref{lem:compact-transporter}. By \Cref{prop:compact-projective-family}, $P_xf(Y(\Cp))$ is closed, so it contains $x$.  Thus $x\in J f(Y(\Cp))$.

Suppose that there is a nonempty open set \(O\subset J f(Y(\Cp)).\)  Choose a strict closed $n$-dimensional rational coordinate polydisc \(B(\Cp)\subset O\) and a compact analytic neighborhood $U$ of $B$. By \Cref{lem:compact-transporter}, the set
\[
P=\{g\in J:gC_0\cap U\ne\varnothing\}
\]
is compact. For \(z\in B(\Cp)\), write \(z=gf(y)\) with \(g\in J\) and \(y\in Y(\Cp)\). Since \(z\in U\), we have \(g\in P\).
Thus \(B(\Cp)\subset P f(Y(\Cp))\).
Since $\dim Y<n$, by \Cref{prop:compact-projective-family}, $P f(Y(\Cp))$ has empty interior. This contradicts $B(\Cp)\subset Pf(Y(\Cp))$ and proves the theorem.
\end{proof}

\section{A dense non-quasi-compact curve in Drinfeld space}
\label{sec:drinfeld-counterexample}

We now remove quasi-compactness from
\Cref{thm:drinfeld-qc}.  The conclusion fails even for a closed smooth curve and one fixed one-parameter subgroup.
Throughout the section, let $F/\mathbb Q_p$ be finite and let $d\geq3$.
\begin{theorem}
\label{thm:drinfeld-nonqc-counterexample}
Define
\(\lambda:F^\times\to\PGL_d(F)\) by
\[\lambda(q)=[\operatorname{diag}(q^{d-1},q^{d-2},\ldots,q,1)].\]
There is a closed smooth connected irreducible non-quasi-compact analytic curve \(A\subset\Omega_F^{d-1}\) such that
\begin{equation}
 \overline{\lambda(F^\times)A(\Cp)}^{\,\mathrm{cl}}=\Omega_F^{d-1}(\Cp)
\label{eq:drinfeld-dense-orbit}.
\end{equation}
In particular, the $\PGL_d(F)$-orbit of $A(\Cp)$ is dense.
\end{theorem}

To construct \(A\), we first find analytic functions on \(\mathbf G_{m,\Cp}^{\an}\) with prescribed values at a locally finite sequence of points.
\begin{lemma}
\label{lem:drinfeld-interpolation}
Let \((x_n)_{n\geq1}\) be a sequence of distinct points in \(\Cp^\times\) such that every strict closed annulus in \(\mathbf G_{m,\Cp}^{\an}\) contains only finitely many \(x_n\).
For every sequence \((b_n)_{n\geq1}\) in \(\Cp\), there exists \(f\in\mathcal O(\mathbf G_{m,\Cp}^{\an})\) such that \(f(x_n)=b_n\) for all \(n\geq1\).
\end{lemma}

\begin{proof}
A rigid analytic space $X$ is called \emph{quasi-Stein} if it has an admissible exhaustion \(X=\bigcup_{n\geq1}X_n\) by affinoid domains with \(X_n\subset X_{n+1}\), such that each restriction map \(\mathcal O(X_{n+1})\to\mathcal O(X_n)\) has dense image. Fix a sequence \((b_n)_{n\geq1}\) in \(\Cp\).
Let \(X=\mathbf G_{m,\Cp}^{\an}\) with coordinate \(T\).
The strict closed annuli
\[
X_m=\{\,|p|^m\leq |T|\leq |p|^{-m}\,\},
\quad m\geq1,
\]
form an admissible affinoid exhaustion of \(X\). Laurent polynomials are dense in \(\mathcal O(X_m)\) and extend to \(X_{m+1}\), so each restriction map \(\mathcal O(X_{m+1})\to\mathcal O(X_m)\) has dense image.
Thus \(X\) is quasi-Stein.

On each \(X_m\), the finitely many points \(x_n\in X_m\) define a reduced divisor whose ideal is generated by \(\prod_{x_n\in X_m}(T-x_n)\).
These ideals are compatible under restriction, since the factors corresponding to points outside \(X_m\) are units on \(X_m\). So they glue and define a closed reduced analytic divisor \(i:D\hookrightarrow X\) with coherent ideal sheaf \(\mathcal I_D\).

Kiehl's Theorem B \cite[Satz~2.4]{Kiehl1967} gives \(H^1(X,\mathcal I_D)=0\).
The long exact cohomology sequence associated with
\[
0\rightarrow\mathcal I_D
\rightarrow\mathcal O_X
\rightarrow i_*\mathcal O_D
\rightarrow0
\]
then shows that \(\mathcal O(X)\to\mathcal O(D)\) is surjective.
Since \(D\) is the locally finite disjoint union of the reduced \(\Cp\)-points \(x_n\), we have \(\mathcal O(D)=\prod_{n\geq1}\Cp\).
Hence \((b_n)_{n\geq1}\) lifts to an analytic function \(f\in\mathcal O(X)\), which satisfies \(f(x_n)=b_n\) for every \(n\geq1\).
\end{proof}

Choose \(f_2,\ldots,f_{d-1} \in\mathcal O(\mathbf G_{m,\Cp}^{\an}).\) Define
\begin{equation*}
 \gamma:\mathbf G_{m,\Cp}^{\an}
       \to\mathbf P_{\Cp}^{d-1,\an},
 \quad
 x\mapsto[x:f_2(x):\cdots:f_{d-1}(x):1],
\end{equation*}
and let \(C_0=\gamma^{-1}(\Omega_F^{d-1})\subset\mathbf G_{m,\Cp}^{\an}\) with its induced open analytic structure.  Write \(A=\gamma(C_0)\subset\Omega_F^{d-1}.\)
\begin{proposition}
\label{prop:drinfeld-graph-geometry}
  
If $C_0$ is nonempty, then $A$ is a closed smooth connected irreducible analytic curve in $\Omega_F^{d-1}$. If $C_0$ contains a sequence $x_n\to0$ in $\mathbf A_{\Cp}^{1,\an}$, then $A$ is not quasi-compact.
\end{proposition}

\begin{proof}
Let \(U=\{X_1X_d\neq0\}\subset\mathbf P_{\Cp}^{d-1,\an}.\)
Then $\Omega_F^{d-1}\subset U$, and \(U\simeq\mathbf G_{m,\Cp}^{\an} \times\mathbf A_{\Cp}^{d-2,\an}.\)
Under this identification, $\gamma$ is the graph of $(f_2,\ldots,f_{d-1})$, hence is a closed immersion into $U$.  Therefore $A$ is the base change of this graph along $\Omega_F^{d-1}\hookrightarrow U$.  It is a closed analytic subspace of $\Omega_F^{d-1}$, and $C_0\to A$ is an isomorphism. Since \(C_0\) is an open analytic subspace of the smooth curve \(\mathbf G_m^{\an}\), both \(A\) and \(C_0\) are smooth curves.

For $a=(a_1,\ldots,a_d)\in F^d\setminus\{0\}$, the pullback of the $F$-rational hyperplane $a_1X_1+\cdots+a_dX_d=0$ is the zero locus of \(g_a(x)=a_1x+\sum_{i=2}^{d-1}a_if_i(x)+a_d.\)
Choose $x_0\in C_0$.  Since $\gamma(x_0)$ does not lie on $F$-rational hyperplanes, $g_a$ is not identically zero. Since $\mathbf G_{m,\Cp}^{\an}$ is connected and normal,  by \cite[Lemma~2.1.4]{ConradIrreducible}, every nonempty admissible open subset is Zariski dense. Hence $g_a$ does not vanish identically on any nonempty annulus. Its zero set on each quasi-compact annulus is either empty or a zero-dimensional affinoid and is a finite set of type I Berkovich points. Since \(\Omega_F^{d-1}\) is the complement of the \(F\)-rational hyperplanes, we have
\begin{equation}
 C_0^{\Ber}
 =(\mathbf P_{\Cp}^{1,\an})^{\Ber}
   \setminus\bigl(S\cup\{0,\infty\}\bigr),
\label{eq:drinfeld-complement-endpoints}
\end{equation}
where \(S\subset(\mathbf G_{m,\Cp}^{\an})^{\Ber}\) is a possibly infinite set of type I points.
The Berkovich projective line is uniquely arc-connected, and its type I points are endpoints \cite[\S1.4.4 and Theorem~4.2.1]{BerkovichSpectral}.
For any two distinct points of \(C_0^{\Ber}\), the arc joining them cannot meet \(S\cup\{0,\infty\}\), since an endpoint cannot lie in the interior of an arc.
Thus this arc lies entirely in \(C_0^{\Ber}\).
It follows that \(C_0^{\Ber}\) is path connected, so \(C_0\) is connected. The space $C_0$ is smooth and hence normal.  In Conrad's definition, the irreducible components are the images of the connected components of the normalization \cite[Definition~2.2.2 and Theorem~2.2.4(1)]{ConradIrreducible}, so connectedness gives only one component. Thus $C_0$ and $A$ are irreducible.

Suppose now that $C_0$ contains a sequence $x_n\to0$.  If $C_0$ were quasi-compact, then $C_0^{\Ber}$ would be compact.  Its image in the Hausdorff Berkovich projective line would be closed, so it would contain $0$.  This is impossible because \(C_0\subset\mathbf G_{m,\Cp}^{\an}\). Since \( A\) is isomorphic to \(C_0\), it is not quasi-compact either.
\end{proof}

\begin{proof}[Proof of \Cref{thm:drinfeld-nonqc-counterexample}]
Every classical point in $\Omega_F^{d-1}$ has a unique expression \([v_1:\cdots:v_{d-1}:1].\) Since \(\Cp\) has a countable dense subset in the classical topology, so does \(\Omega_F^{d-1}(\Cp)\). Choose a dense sequence
\[
 u_n=[v_{1,n}:\cdots:v_{d-1,n}:1]
       \in\Omega_F^{d-1}(\Cp).
\]
The hyperplane $X_1=0$ is also deleted, so $v_{1,n}\neq0$.

Fix a uniformizer $\varpi$ of $F$.  Choose integers $N_n\to\infty$ such that \(x_n=\varpi^{(d-1)N_n}v_{1,n}\in\Cp^\times\) have strictly decreasing absolute values and tend to $0$. For example, choose \(N_1\) so that \(|x_1|<|\varpi|\).
For each \(n\geq2\) choose \(N_n>N_{n-1}\) such that \(|x_n|<\min\{|x_{n-1}|,|\varpi|^n\}\). The sequence is locally finite on $\mathbf G_{m,\Cp}^{\an}$.  By \Cref{lem:drinfeld-interpolation}, for each $2\leq i\leq d-1$ there is an $f_i\in\mathcal O(\mathbf G_{m,\Cp}^{\an})$ such that \(f_i(x_n)=\varpi^{(d-i)N_n}v_{i,n}.\)
Write $\gamma$, $C_0$, and $A$ as in \Cref{prop:drinfeld-graph-geometry}. Let $q_n=\varpi^{-N_n}$.  Then \(\lambda(q_n)\gamma(x_n)=u_n.\)
Since $\Omega_F^{d-1}$ is $\PGL_d(F)$-stable, $\gamma(x_n)=\lambda(q_n)^{-1}u_n$ belongs to $\Omega_F^{d-1}$.  Hence $x_n\in C_0$.  The proposition shows that $A$ is closed smooth connected irreducible and non-quasi-compact. The density of $(u_n)$ then gives \eqref{eq:drinfeld-dense-orbit}.
\end{proof}

\section{An open disc with dense orbit}
\label{sec:qp-gauss-counterexample}

Let
\[
 \mathbb D^-_{\mathbb Q_p}
 =\{x\in\mathbf A_{\mathbb Q_p}^{1,\an}:|T(x)|<1\}
\]
be the rigid open unit disc.  We construct a map from \(\mathbb D^-_{\mathbb Q_p}\) defined over $\mathbb Q_p$ whose $\PGL_3(\mathbb Q_p)$-orbit is dense after base change to $\Cp$.  We write $v$ for the valuation on $\Cp$ normalized by $v(p)=1$.

\begin{theorem}
\label{thm:qp-disc-dense-orbit}
There exist power series \(A(T),B(T)\in\mathbb Q_p[[T]]\) which converge for every $|T|<1$, and an analytic morphism
\[
 f:\mathbb D^-_{\mathbb Q_p}\to\mathbf P^2_{\mathbb Q_p},
 \quad f(T)=[A(T):B(T):1],
\]
such that
\begin{equation}
 \overline{
   \PGL_3(\mathbb Q_p)f(\mathbb D^-_{\mathbb Q_p}(\Cp))
 }^{\,\mathrm{cl}}
 =\mathbf P^2(\Cp).
 \label{eq:qp-disc-dense-orbit}
\end{equation}
Both $A$ and $B$ are unbounded on the open disc.  Consequently the two affine coordinates do not extend to functions on $\operatorname{Spf}\mathbb Z_p[[T]]$ after multiplication by any common power of $p$. Here
$\mathbb Z_p[[T]]$ carries its $(p,T)$-adic topology.
\end{theorem}
The power series \(A,B\) have coefficients in \(\mathbb Q_p\). How can their values become dense in \(\Cp^2\) after applying elements of \(\PGL_3(\mathbb Q_p)\)?
We interpolate suitable scalar multiples of a dense sequence of algebraic points in \(\Cp^2\) at closed points of the disc whose residue fields contain the prescribed values. This gives analytic functions defined over \(\mathbb Q_p\), while diagonal projective transformations recover the original dense sequence. We represent these closed points by elements $t_j\in\overline{\mathbb Q}_p$ with $|t_j|<1$ and $|t_j|\to1$. Thus every closed subdisc of radius less than one contains only finitely many of them.

\subsection{Closed points approaching the Gauss point}
\begin{lemma}
\label{lem:qp-primitive-small-valuation}
Let $E/\mathbb Q_p$ be finite and let $\epsilon>0$. There are a finite extension $L/\mathbb Q_p$ containing $E$ and an element $t\in L$ such that \(L=\mathbb Q_p(t)\) with \(0<v(t)<\epsilon.\)
The element $t$ may be chosen to be a uniformizer of $L$ and all its $\mathbb Q_p$-conjugates in $\Cp$ then have the same absolute value.
\end{lemma}

\begin{proof}
Let $e_E=e(E/\mathbb Q_p)$, choose a uniformizer $\pi_E$ of $E$, and choose an integer $m$ with \(\frac{1}{me_E}<\epsilon.\)
The polynomial $X^m-\pi_E$ is Eisenstein over $E$.  If $\varpi$ is one of its roots and $M=E(\varpi)$, then $M/E$ is totally ramified of degree $m$ and $e(M/\mathbb Q_p)=me_E$.

There are only finitely many intermediate fields between $\mathbb Q_p$ and $M$ which are controlled by the subgroups of the finite Galois group of a Galois closure.  Every proper intermediate field is a proper closed $\mathbb Q_p$-linear subspace of $M$ and has empty interior.  The set of uniformizers of $M$ is the nonempty open subset
\[
        \{x\in M:v(x)=1/e(M/\mathbb Q_p)\}
\]
so it cannot be contained in the finite union of the proper intermediate fields.  Choose a uniformizer $t$ outside that union.  Then \(\mathbb Q_p(t)=M\) and \(v(t)=\frac{1}{me_E}<\epsilon.\)
Every embedding of a finite extension of $\mathbb Q_p$ into $\Cp$ preserves the unique extension of the $p$-adic absolute value.  The final assertion follows.
\end{proof}

\begin{lemma}
\label{lem:qp-locally-finite-points}
Let \((a_j,b_j)_{j\geq1}\) be a sequence in \((\overline{\mathbb Q}_p^\times)^2\).  There are finite extensions $L_j/\mathbb Q_p$ and elements $t_j\in L_j\subset \Cp$ such that:
\begin{enumerate}
\item $a_j,b_j\in L_j$ and $L_j=\mathbb Q_p(t_j)$;
\item $0<|t_j|<1$, the numbers $|t_j|$ are strictly increasing, and $|t_j|\to1$;
\item the closed points $z_j\in\mathbb D^-_{\mathbb Q_p}$ defined by the minimal polynomial of \(t_j\) over \(\mathbb Q_p\) form a closed locally finite reduced analytic subspace \(Z=\coprod_{j\geq1}z_j\) and the residue field of each $z_j$ is $L_j$.
\end{enumerate}
\end{lemma}

\begin{proof}
Choose a finite extension $E_j/\mathbb Q_p$ containing $a_j$ and $b_j$.
Apply \Cref{lem:qp-primitive-small-valuation} successively to obtain $L_j\supset E_j$ and primitive uniformizers $t_j$ whose positive valuations strictly decrease to zero. This proves the first two assertions.

The kernel of the map \(\mathbb Q_p[T]\to L_j\) sending \(T\) to \(t_j\) is generated by the minimal polynomial of $t_j$, so the closed point has residue field $L_j$.  All roots of that polynomial in $\Cp$ have absolute value $|t_j|$.
Hence every member of the cofinal strict-affinoid exhaustion \(\{|T^m|\leq|p|\}_{m\geq1} \) meets only finitely many $z_j$.  Their union is locally finite.  Since the open disc is smooth of dimension one, each $z_j$ is an effective Cartier divisor.
A locally finite union of the closed points \(z_j\) is again a closed reduced analytic subspace with a coherent ideal sheaf.
\end{proof}

\subsection{Quasi-Stein interpolation and the dense orbit}

The open disc is quasi-Stein, since it is exhausted by the strict rational affinoids \(X_m=\{\,|T^m|\leq |p|\,\},\) and  polynomials are dense in $\mathcal O(X_m)$ and extend to $X_{m+1}$.

\begin{lemma}
\label{lem:qp-kiehl-evaluation}
For the subspace $Z=\coprod_jz_j$ in \Cref{lem:qp-locally-finite-points}, the evaluation map
\[
 \mathcal O(\mathbb D^-_{\mathbb Q_p})
   \to
 \mathcal O(Z)=\prod_{j\geq1}L_j
\]
is surjective.
\end{lemma}

\begin{proof}
Let $\mathcal I_Z$ be the coherent ideal sheaf of $Z$.
Apply Kiehl's Theorem~B as in the proof of \Cref{lem:drinfeld-interpolation}. The vanishing of $H^1(\mathbb D^-_{\mathbb Q_p},\mathcal I_Z)$ makes the restriction map to $\mathcal O(Z)$ surjective. Since $Z=\coprod_j\operatorname{Sp}L_j$, we have $\mathcal O(Z)=\prod_{j\geq1}L_j$.
\end{proof}

\begin{lemma}
\label{lem:qp-disc-unboundedness}
Let
\[
h(T)=\sum_{m\geq0}c_mT^m \in\mathbb Q_p[[T]]
\]
converge for every $|T|<1$. Suppose that $h\neq0$ and that there is a sequence $t_j\in\Cp$ with $|t_j|<1$, $|t_j|\to1$, and \(h(t_j)\to0.\)
Then $h$ is unbounded on the open unit disc.
\end{lemma}

\begin{proof}
Suppose that $h$ is bounded. The nonarchimedean Cauchy estimate on closed subdiscs shows that the valuations of the coefficients $c_m$ are bounded below. Since the value group of $\mathbb Q_p$ is discrete, $\sup_m|c_m|$ is attained. Let $e$ be the first index at which it is attained. For $|t|<1$ sufficiently close to $1$, the term $c_et^e$ is then the unique term of maximal absolute value, and hence \(|h(t)|=|c_e|\,|t|^e.\) 
This cannot tend to zero as $|t|\to1$, contradicting $h(t_j)\to0$.
\end{proof}

\begin{proof}[Proof of \Cref{thm:qp-disc-dense-orbit}]
The field \(\Cp\) has a countable dense subset contained in \(\overline{\mathbb Q}_p\).
Choose a sequence \((a_j,b_j)_{j\geq1}\) in \((\overline{\mathbb Q}_p^\times)^2\) that is dense in \(\Cp^2\).
Apply \Cref{lem:qp-locally-finite-points} to this sequence to obtain \(L_j\) and \(t_j\).
Choose strictly increasing integers $N_j$ such that \(|p|^{N_j}\max\{1,|a_j|,|b_j|\}\leq |p|^j.\)
This is possible because $|p|^N\to0$ as $N\to\infty$.  In particular, \(p^{N_j}a_j\rightarrow0\) and \(p^{N_j}b_j\rightarrow0.\)
By \Cref{lem:qp-kiehl-evaluation} the two elements \((p^{N_j}a_j)_j,\ (p^{N_j}b_j)_j\in\prod_jL_j\) lift to analytic functions $A,B\in\mathcal O(\mathbb D^-_{\mathbb Q_p})$. Since global analytic functions on the open disc have unique power series expansions over $\mathbb Q_p$ which converge for $|T|<1$,  $A,B$ have the form stated in \cref{thm:qp-disc-dense-orbit} and
\begin{equation}
 A(t_j)=p^{N_j}a_j, \quad  B(t_j)=p^{N_j}b_j.
\label{eq:qp-interpolation-values}
\end{equation}
Define \(f(T)=[A(T):B(T):1]\).
Let
\[
 g_j=\operatorname{diag}(p^{-N_j},p^{-N_j},1) \in\PGL_3(\mathbb Q_p).
\]
Equation \eqref{eq:qp-interpolation-values} gives \(g_jf(t_j)=[a_j:b_j:1].\)
These points are dense in the affine chart $\{Z\neq0\}\simeq \Cp^2$, and the chart itself is dense in $\mathbf P^2(\Cp)$. So this proves
\eqref{eq:qp-disc-dense-orbit}.

We have
\(A(t_j)=p^{N_j}a_j\rightarrow0\) and \(B(t_j)=p^{N_j}b_j\rightarrow0\) while $|t_j|\to1$. Since every $a_j$ and $b_j$ is nonzero, both $A$ and $B$ are nonzero. \Cref{lem:qp-disc-unboundedness} shows that $A$ and $B$ are unbounded on the open disc.

A global function on $\operatorname{Spf}\mathbb Z_p[[T]]$ is an element of $\mathbb Z_p[[T]]$ and has absolute value at most one on $|T|<1$.
Thus if $p^N A$ belonged to $\mathbb Z_p[[T]]$ for any integer $N$, then $A$ would be bounded by $|p|^{-N}$ on the open disc, contradiction.
The same holds for $B$. 
\end{proof}
   
\section{A perfectoid torus with orbit closure of nonempty interior}
\label{sec:perfectoid-torus-counterexample}

The following example has a quasi-compact one-dimensional domain, but the domain is perfectoid rather than a strict rigid space of finite type. 

Denote by $\Cp\langle T^{\pm1/p^\infty}\rangle$ the completion of the union of the finite-level algebras $\Cp\langle T^{\pm1/p^m}\rangle$ for the Gauss norm.

\begin{theorem}
\label{thm:perfectoid-torus-counterexample}
There exists a connected quasi-compact affinoid perfectoid space
\[
 P=\operatorname{Spa}
 \bigl(\Cp\langle T^{\pm1/p^\infty}\rangle,
       \Cp\langle T^{\pm1/p^\infty}\rangle^\circ\bigr)
 \sim\varprojlim_m
 \operatorname{Spa}\Cp\langle T_m^{\pm1}\rangle,
 \quad T_m=T^{1/p^m},
\]
whose underlying spectral topological space has Krull dimension one and
whose finite-level transition maps are finite \'etale, together with an
analytic morphism
\(g:P\to\mathbf P^2_{\Cp}\)
such that
\(\overline{\PGL_3(\mathbb Q_p)g(P(\Cp))}^{\,\mathrm{cl}}\)
contains a nonempty open subset of $\mathbf P^2(\Cp)$.
\end{theorem}

We construct the coordinate functions as convergent series using successively deeper levels of the tower. We choose points and diagonal transformations so that one term of each series gives the prescribed value and the remaining terms give an error tending to zero; see \Cref{prop:perfectoid-torus-estimates}.
By choosing the prescribed values to be dense in an open set and with every tail also dense, we obtain an orbit closure containing that open set.

\subsection{The perfectoid torus}
For the completed algebra \(\Cp\langle T^{\pm1/p^\infty}\rangle\) and its adic spectrum, write \(\Cp\langle T^{\pm1/p^\infty}\rangle^+=\Cp\langle T^{\pm1/p^\infty}\rangle^\circ,\) 
\[P=\operatorname{Spa}(\Cp\langle T^{\pm1/p^\infty}\rangle,\Cp\langle T^{\pm1/p^\infty}\rangle^+),\]
and define the finite-level unit torus $P_m$ for each $m\geq0$ by
\[
 P_m=
 \operatorname{Spa}
 \bigl(\Cp\langle T_m^{\pm1}\rangle,
       \Cp\langle T_m^{\pm1}\rangle^\circ\bigr),
 \quad
 T_m=T_{m+1}^p.
\]
\begin{lemma}
\label{lem:perfectoid-torus-tower}
The space $P$ is a quasi-compact affinoid perfectoid space, every map
$P_{m+1}\to P_m$ is finite \'etale, and
\(P\sim\varprojlim_mP_m\)
in the sense of Scholze--Weinstein
\cite[Definition~2.4.1]{ScholzeWeinstein}.
\end{lemma}

\begin{proof}
By the example following Theorem~1.7 in
\cite{ScholzePerfectoid}, the algebra \(\Cp\langle T^{1/p^\infty}\rangle\) is perfectoid. Consider the rational subset \(\{x: 1=|1(x)|\leq|T(x)|\}\) in its adic spectrum.
Since $|T|\leq1$ on the closed perfectoid disc, this is the locus $|T|=1$.  Its coordinate algebra is \(\Cp\langle T^{\pm1/p^\infty}\rangle\), and it is perfectoid by \cite[Theorem~6.3(ii)]{ScholzePerfectoid}. Thus $P$ is affinoid perfectoid and in particular quasi-compact.

The map $P_{m+1}\to P_m$ is induced by
\[
 \Cp\langle T_m^{\pm1}\rangle \to \Cp\langle T_{m+1}^{\pm1}\rangle, \quad T_m\mapsto T_{m+1}^p.
\]
Equivalently we have
\[
 \Cp\langle T_{m+1}^{\pm1}\rangle  \simeq \Cp\langle T_m^{\pm1}\rangle\langle X\rangle/(X^p-T_m).
\]
The algebra on the right is finite free over $\Cp\langle T_m^{\pm1}\rangle$ with basis \(\{1,X,\ldots,X^{p-1}\}\). 
The derivative of $X^p-T_m$ is $pX^{p-1}$. Since both $p$ and $X=T_{m+1}$ are units in $\Cp\langle T_{m+1}^{\pm1}\rangle$, $pX^{p-1}$ is a unit. Thus $P_{m+1}\to P_m$ is finite \'etale.

Let
\[
 A_{\infty,0} =\varinjlim_m\Cp^\circ\langle T_m^{\pm1}\rangle, \quad  A_\infty=A_{\infty,0}[1/p], \quad A_\infty^+=A_{\infty,0}.
\]
Give $A_\infty$ the f-adic topology for which $A_{\infty,0}$ is an open ring of definition and $(p)$ is an ideal of definition.  The rings $\Cp^\circ\langle T_m^{\pm1}\rangle$ and the common ideals $(p)$ satisfy the hypotheses of \cite[Proposition~2.4.2]{ScholzeWeinstein}.  Therefore \(\operatorname{Spa}(A_\infty,A_\infty^+)\sim\varprojlim_mP_m.\)

The closure of $A_{\infty,0}$ for the Gauss norm on $\Cp\langle T^{\pm1/p^\infty}\rangle$ is $\Cp\langle T^{\pm1/p^\infty}\rangle^\circ$.
Consequently, the strong completion of $(A_\infty,A_\infty^+)$ is $(\Cp\langle T^{\pm1/p^\infty}\rangle,\Cp\langle T^{\pm1/p^\infty}\rangle^+)$. 

Since $(\Cp\langle T^{\pm1/p^\infty}\rangle,\Cp\langle T^{\pm1/p^\infty}\rangle^+)$ is perfectoid, $\operatorname{Spa}(A_\infty,A_\infty^+)$ is preperfectoid in the sense of \cite[Definition~2.3.4]{ScholzeWeinstein}.  By Proposition~2.4.4 of \cite{ScholzeWeinstein}, \(P\sim\varprojlim_mP_m.\)
\end{proof}

\begin{lemma}
\label{lem:perfectoid-torus-topology}
The space $P$ is connected, and the Krull dimension of its underlying spectral topological space equals one.
\end{lemma}

\begin{proof}
The tilt of $P$ is
\[
 P^\flat=
 \operatorname{Spa}
 \bigl(\Cp^\flat\langle T^{\pm1/p^\infty}\rangle,
       \Cp^{\flat\circ}\langle T^{\pm1/p^\infty}\rangle\bigr).
\] Here $\Cp^\flat$ denotes the tilt of $\Cp$.
By \cite[Theorem~6.3(i)]{ScholzePerfectoid}, the underlying topological spaces of $P$ and $P^\flat$ are homeomorphic.

In characteristic $p$, the algebra \(\Cp^\flat\langle T^{\pm1/p^\infty}\rangle\) is the completed perfection of \(\Cp^\flat\langle T^{\pm1}\rangle.\)
By \cite[Proposition~6.11(i)]{ScholzePerfectoid}, their adic spectra are homeomorphic. Hence $P$ is homeomorphic to
\[
 Q=
 \operatorname{Spa}
 \bigl(\Cp^\flat\langle T^{\pm1}\rangle,
       \Cp^{\flat\circ}\langle T^{\pm1}\rangle\bigr).
\]

The ring $\Cp^\flat\langle T^{\pm1}\rangle$ is a domain and therefore has no nontrivial idempotent.  A decomposition of $Q$ into two nonempty clopen subsets gives a nontrivial idempotent in its ring of global functions.  Thus $Q$ is connected and so $P$ is connected.

The affinoid algebra $\Cp^\flat\langle T^{\pm1}\rangle$ has Krull dimension one.  By \cite[Lemma~1.8.6(ii)]{HuberEtale}, the underlying spectral topological space of $Q$ also has Krull dimension one. The homeomorphisms from $P$ to $P^\flat$ and from $P^\flat$ to $Q$ therefore give the asserted dimension of $P$.
\end{proof}

\subsection{Points and approximation estimates}
We also write \(T_m\) for the pullback of the coordinate on \(P_m\) along \(P\to P_m\) and define \(S_m:=T_m-1\in\mathcal O(P)\).

\begin{lemma}
\label{lem:perfectoid-torus-points}
Let \(0\leq k_1<k_2<\cdots\) be integers, and let $c$ be an integer satisfying \(c>\frac1{p-1}.\)
For every $j$, there is a point $x_j\in P(\Cp)$ such that \(T_{k_j}(x_j)=1+p^c\) and these points satisfy
\begin{align}
 v(S_{k_l}(x_j))
   &=c+k_j-k_l &&l<j,
 \label{eq:perfectoid-torus-earlier-valuation}\\
 v(S_{k_l}(x_j))
   &\geq0 && l>j.
 \label{eq:perfectoid-torus-later-valuation}
\end{align}
\end{lemma}

\begin{proof}
Fix \(j\geq1\). Let \(t_{k_j}=1+p^c\) and \(t_m=(1+p^c)^{p^{k_j-m}}\) for \(m<k_j\).
For \(m>k_j\), choose \(t_m\in\Cp\) recursively so that \(t_m^p=t_{m-1}\). 
All \(t_m\) have absolute value \(1\), so the assignments \(T_m\mapsto t_m\) define compatible evaluations on the finite-level algebras, which is  bounded for the Gauss norm.
Therefore they extend continuously to \(\Cp\langle T^{\pm1/p^\infty}\rangle\) and define a point \(x_j\in P(\Cp)\) with \(T_{k_j}(x_j)=1+p^c\).

Suppose that $l<j$.  Then
\[
 T_{k_l}(x_j)
   =(1+p^c)^{p^{k_j-k_l}}
   =\exp\bigl(p^{k_j-k_l}\log(1+p^c)\bigr).
\]
The assumption $c>1/(p-1)$ gives \(v(\log(1+p^c))=c\) and \(v(\exp(z)-1)=v(z)\) when \(v(z)>\frac1{p-1}.\)
It follows that
\[
 v(T_{k_l}(x_j)-1)=c+k_j-k_l.
\]

If $l>j$, then $|T_{k_l}(x_j)|=1$. The ultrametric inequality gives \(|T_{k_l}(x_j)-1|\leq 1\) which proves \eqref{eq:perfectoid-torus-later-valuation}.
\end{proof}

\begin{proposition}
\label{prop:perfectoid-torus-estimates}
Assume that the $k_j$ are even and that \(k_{j+1}-k_j\to+\infty.\)
Let \(A_j=\frac{k_j}{2}\) and \(n_j=A_j+c.\)
Let \((v_j,w_j)\in(\mathcal O_{\Cp}^{\times})^2\) be any sequence.  Define \(h_1=\sum_{j\geq1}p^{A_j}w_jS_{k_j},\) and \(h_2=\sum_{j\geq1}p^{2A_j}v_jS_{k_j}^2\).
Then these series converge in $\Cp\langle T^{\pm1/p^\infty}\rangle$.  Moreover, there are $\varepsilon_{1,j},\varepsilon_{2,j}\in\Cp$ such that
\begin{align}
 p^{-n_j}h_1(x_j)
   &=w_j+\varepsilon_{1,j},
 & \varepsilon_{1,j}&\to0,
 \label{eq:perfectoid-torus-h2-limit}\\
 p^{-2n_j}h_2(x_j)
   &=v_j+\varepsilon_{2,j},
 & \varepsilon_{2,j}&\to0.
 \label{eq:perfectoid-torus-h1-limit}
\end{align}
\end{proposition}

\begin{proof}
Since $T_m$ and $T_m^{-1}$ are power-bounded we have \(\lVert S_m\rVert\leq1.\)
The coefficients $p^{A_j}$ and $p^{2A_j}$ tend to zero.  Therefore \(h_1\) and \(h_2\) converge in the complete Banach algebra $\Cp\langle T^{\pm1/p^\infty}\rangle$.

At $x_j$, the $j$-th summand of $h_1$ is
\[
 p^{A_j}w_jS_{k_j}(x_j)
   =p^{A_j+c}w_j
   =p^{n_j}w_j.
\]
Define
\[
 \varepsilon_{1,j}
   =
   p^{-n_j}
   \sum_{l\neq j}p^{A_l}w_lS_{k_l}(x_j).
\]
For $l<j$, equation
\eqref{eq:perfectoid-torus-earlier-valuation} gives
\begin{align*}
 v(p^{A_l}w_lS_{k_l}(x_j))-n_j
 =A_l+c+k_j-k_l-(A_j+c)
 =\frac{k_j-k_l}{2}
 \geq\frac{k_j-k_{j-1}}2.
\end{align*}
For $l>j$, equation
\eqref{eq:perfectoid-torus-later-valuation} gives
\begin{align*}
 v(p^{A_l}w_lS_{k_l}(x_j))-n_j
 \geq A_l-A_j-c
 =\frac{k_l-k_j}{2}-c
 \geq\frac{k_{j+1}-k_j}{2}-c.
\end{align*}
For $j\geq2$, the ultrametric inequality therefore gives
\[
 v(\varepsilon_{1,j})
 \geq
 \min\left\{
       \frac{k_j-k_{j-1}}2,\,
       \frac{k_{j+1}-k_j}2-c
     \right\}.
\]
The right hand side tends to $+\infty$, so $\varepsilon_{1,j}\to0$. This proves \eqref{eq:perfectoid-torus-h2-limit}.

At $x_j$, the $j$-th summand of $h_2$ is
\[
 p^{2A_j}v_jS_{k_j}(x_j)^2
   =p^{2A_j+2c}v_j
   =p^{2n_j}v_j.
\]
Define
\[
 \varepsilon_{2,j}
   =
   p^{-2n_j}
   \sum_{l\neq j}p^{2A_l}v_lS_{k_l}(x_j)^2.
\]
Since $v_l$ and $w_l$ are units, the valuation of each normalized summand in $\varepsilon_{2,j}$ is twice that of the corresponding summand in $\varepsilon_{1,j}$. The preceding estimates and the ultrametric inequality therefore give
\[
 v(\varepsilon_{2,j}) \geq \min\left\{k_j-k_{j-1},\, k_{j+1}-k_j-2c \right\}
\]  for $j\geq2$.
This lower bound tends to $+\infty$, so $\varepsilon_{2,j}\to0$. This proves \eqref{eq:perfectoid-torus-h1-limit}.
\end{proof}

\subsection{Proof of \Cref{thm:perfectoid-torus-counterexample}}
\label{subsec:perfectoid-orbit-closure}

\begin{proof}
The space $P$ satisfies all the asserted properties by \Cref{lem:perfectoid-torus-tower,lem:perfectoid-torus-topology}.

Choose \(k_j=2j^2,\) and fix an integer $c>1/(p-1)$. The metric space $(\mathcal O_{\Cp}^{\times})^2$ is separable and has no isolated points. So we can choose a dense sequence \((v_j,w_j)\in(\mathcal O_{\Cp}^{\times})^2\).
Every tail of this sequence is dense, so
\begin{equation}
 \overline{\{(v_j,w_j):j\geq N\}}^{\,\mathrm{cl}}
      =(\mathcal O_{\Cp}^{\times})^2
 \quad\text{for every }N\geq1.
 \label{eq:perfectoid-torus-dense-tails}
\end{equation}
Apply \Cref{lem:perfectoid-torus-points,prop:perfectoid-torus-estimates}.
Since the last homogeneous coordinate is $1$, \(g=[h_2:h_1:1]\) defines an analytic morphism \(g:P\to\mathbf P^2_{\Cp}.\)

Since $n_j=A_j+c\in\mathbb Z$, we can take
\[
 g_j= \operatorname{diag}(p^{-2n_j},p^{-n_j},1)
 \in\PGL_3(\mathbb Q_p).
\]
Equations \eqref{eq:perfectoid-torus-h2-limit} and \eqref{eq:perfectoid-torus-h1-limit} give
\[
 g_j g(x_j) =[v_j+\varepsilon_{2,j}: w_j+\varepsilon_{1,j}:1]. 
\]

Let
\[
 \mathcal U=
 \{[v:w:1]\in\mathbf P^2(\Cp):|v|=|w|=1\}.
\]
Fix $[v:w:1]\in\mathcal U$ and a classical neighborhood \(W\) of this point.  Choose $\rho>0$ such that
\[
 \{[v':w':1]:
       |v'-v|<\rho,\ |w'-w|<\rho\}
 \subset W.
\]
Choose \(N\) such that
\(|\varepsilon_{1,j}|<\rho,\) \(|\varepsilon_{2,j}|<\rho\) for \(j\geq N\).
By \eqref{eq:perfectoid-torus-dense-tails}, there is a \(j\geq N\) such that \(|v_j-v|<\rho\) and \(|w_j-w|<\rho.\)
The ultrametric inequality then gives \(g_j g(x_j)\in W.\) Thus \(\mathcal U\subset \overline{\PGL_3(\mathbb Q_p)g(P(\Cp))}^{\,\mathrm{cl}}.\)
Since $\mathcal U$ is nonempty and open, the theorem follows.
\end{proof}


\end{document}